\documentclass[11pt]{article}

\usepackage{amsthm,amsfonts,amssymb,amsmath,color}
\usepackage{bm}
\usepackage{bbm}
\usepackage[utf8x]{inputenc}
\usepackage{enumerate}
\usepackage[english]{babel}
\usepackage{esint}
\usepackage[mathscr]{eucal}
\usepackage[numbers,sort&compress]{natbib}
\usepackage{hyperref}
\hypersetup{
	colorlinks = true, 
	urlcolor = blue, 
	linkcolor = blue, 
	citecolor = blue 
}

\usepackage{url}

\usepackage{layout}
\allowdisplaybreaks

\numberwithin{equation}{section}

\renewcommand\d{\partial}
\renewcommand\a{\alpha}
\renewcommand\b{\beta}

\newcommand\s{\sigma}

\newcommand\R{\mathbb R}

\def\O{\Omega}

\def\l{\lambda}

\def\bvp{\bm{\varphi}}

\def\bsp{\bm{\psi}}
\def\e{\varepsilon}

\newcommand\br{\begin{rem}}
	\newcommand\er{\end{rem}}
\newcommand\bp{\begin{pmatrix}}
	\newcommand\ep{\end{pmatrix}}
\newcommand\be{\begin{equation}}
	\newcommand\ee{\end{equation}}
\newcommand\ba{\begin{equation}\begin{aligned}}
		\newcommand\ea{\end{aligned}\end{equation}}

\newcommand\nn{\nonumber}

\usepackage{geometry}
\newcommand{\RR}{{\mathbb R}}

\newcommand{\GG}{{\mathbb G}}

\newcommand{\uu}{{\mathbf u}}

\newcommand{\Ov}[1]{\overline{#1}}

\newcommand{\vr}{\varrho}

\newcommand{\vu}{\vc{u}}

\newcommand{\vc}[1]{{\bf #1}}

\newcommand{\Grad}{\nabla_x}

\newcommand{\dx}{\, {\rm d} {x}}
\newcommand{\dt}{\, {\rm d} t }

\newcommand{\intO}[1]{\int_{\O} #1 \ \dx}
\newcommand{\intOe}[1]{\int_{\O_\e} #1 \ \dx}

\newcommand{\dive}{{\rm div}_x}

\newtheorem{defi}{Definition}[section]
\newtheorem{theorem}[defi]{Theorem}
\newtheorem{proposition}[defi]{Proposition}
\newtheorem{lemma}[defi]{Lemma}

\newtheorem{remark}[defi]{Remark}

\newcommand{\calB}{\mathcal{B}}
\renewcommand{\l}{\langle}
\renewcommand{\r}{\rangle}

\begin{document}
	
	\title{Homogenization of a non-homogeneous incompressible heat-conducting fluid in perforated domains}
	
	\author{Yong Lu\footnote{School of Mathematics, Nanjing University, Nanjing 210093, China, luyong@nju.edu.cn}\and Jiaojiao Pan\footnote{Institute of Applied Physics and Computational Mathematics, Beijing 100088, China, panjiaojiao.math@gmail.com}\and Luqi Wang\footnote{School of Mathematics, Nanjing University, Nanjing 210093, China, wangluqi@nju.edu.cn}}
	
	\date{}
	
	\maketitle
	
	\renewcommand{\refname}{References}
	
	\begin{abstract}
		This paper provides the study of the homogenization of the 3D non-homogeneous incompressible heat-conducting fluid in perforated domains with holes of subcritical size, where the viscosity and the heat conductivity coefficient are assumed to depend on the temperature. The diameter of the holes is of order $\varepsilon^{\alpha} \ (\alpha>3)$, where $\varepsilon > 0$ is a small parameter that measures the mutual distance between the holes. We prove that as $\varepsilon\to 0$, the limit behavior of velocity, density and temperature is governed by the original system in homogeneous domain without holes.
			\end{abstract}
	
	{\bf Keywords.} Homogenization; perforated domains; non-homogeneous incompressible heat-conducting fluid.
	\par{\bf Mathematics Subject Classification.} 35B27, 76M50, 80M40.

	
	
	\section{Introduction}\label{INTRODUCTION}
	In this paper, we study the homogenization of the non-homogeneous heat conduction equations in perforated domains containing a large number of solid obstacles (named holes) in $\mathbb{R}^{3}$. Our goal is to describe the limit behavior of weak solutions to the non-homogeneous heat conduction equations as the number of holes tends to infinity and the size of the holes tends to zero simultaneously. This process is called homogenization, and it gives rise to the limit system of fluid flow, called the homogenized system.
	
	Homogenization of various types of fluid flows in physical domains has been widely studied during the last decades. Darcy’s phenomenological law can be traced back to his experimental studies in \cite{Darcy1856}. Later, Tartar \cite{Tartar1} mathematically considered the incompressible stationary Stokes equations, where the size of holes is proportional to the mutual distance between holes, and derived Darcy's law. In the 1990s, Allaire \cite{ALL-NS1, ALL-NS2} systematically studied steady Stokes and Navier-Stokes equations in $\R^d (d \geq 2)$, and proved that the homogenized system depends on the ratio $\s_{\e}$, which is defined by
\ba \label{sigma}
\sigma _\e: = \left(\frac{\e^d}{a_\e^{d-2}}\right)^{\frac{1}{2}}, \ d \geqslant 3;\quad{\sigma _\e}: = \e \left| \log \frac{{a_\e }}{\e} \right|^{\frac{1}{2}}, \ d = 2.
\ea
Here $\e > 0$ denotes the mutual distance between the holes, and $a_\e > 0$ represents their radius. More precisely, $\lim_{\e\to 0}\s_{\e}=0$ corresponds to supercritical case, in which case the holes are relatively large, the fluid velocity eventually tends to zero, and the limit system is governed by Darcy’s law; in subcritical case, i.e., $\lim_{\e\to 0}\s_{\e}=\infty$, the holes are tiny and the limit system coincides with the original one; while in critical case, i.e., $\lim_{\e\to 0}\s_{\e}=\s_{*}\in (0,+\infty)$, the limit system is given by Brinkman's law, which combines the original system with an additional friction term, see \cite{HoeferKowalczykSchwarzacher2021} for a heuristic explanation. The first author \cite{L} proposed a unified approach to all the regimes mentioned above.
	
	Regarding compressible fluids, Masmoudi \cite{Mas-Hom} obtained the first homogenization result for the Navier-Stokes system in the regime where the hole size is proportional to the mutual distance between holes, and he derived Darcy's law. In the same regime, H\"{o}fer, Ne\v{c}asov\' {a} and Oschmann \cite{HNO2026} proved quantitative homogenization toward Darcy’s law via a relative energy method in periodically perforated domains. For the subcritical case in three spatial dimensions, related works \cite{BO2,DFL,FL1,Lu-Schwarz18, OschmannPokorny2023} showed that the limit system coincides with the original one and \cite{NP,NO} complemented similar results for two-dimensional case.  At low Mach number, Bella and Oschmann \cite{BO1} as well as H\"ofer, Kowalczyk and Schwarzacher \cite{HoeferKowalczykSchwarzacher2021} investigated the critical and supercritical case respectively. 
	
	For the homogenization of the evolutionary incompressible Navier-Stokes system, the case where the size of holes is proportional to mutual distance was considered by Mikeli\'{c} \cite{Mik91}, and he derived Darcy's law. The critical, subcritical, and supercritical regimes were considered in \cite{FeNaNe, LY}. As for non-homogeneous incompressible fluids, the supercritical, subcritical and critical case were considered in \cite{BOP26}, \cite{LuPanYang2025} and \cite{Pan2025}, respectively.
		
	Moreover, a series of complicated models describing fluid flows in homogenization have been investigated: Feireisl, Novotn\'y and Takahashi \cite{FNT-Hom} studied the full Navier-Stokes-Fourier system in the supercritical regime, Basari{\'c} and Chaudhuri \cite{BCh24} studied the subcritical regime; the first author and Pokorn\'y \cite{Lu-Pokorny} considered the subcritical case of stationary compressible Navier-Stokes-Fourier equations. Recently, the first author and Qian \cite{LY-Qian} considered the homogenization of evolutionary incompressible viscous non-Newtonian flows of Carreau type, where the size of holes is proportional to their mutual distance. For supercritical case, Darcy's law was derived in \cite{HLO2025} together with quantitative convergence rates. 
		
	As for non-homogeneous incompressible heat-conducting fluid in this paper, Feireisl, the first author and Sun \cite{FLS} investigated the critical case confined to a 3D domain and derived that the limit system is given by Brinkman's law. In the present paper, we wish to complete the study by investigating the homogenization of the 3D non-homogeneous heat-conducting fluid, where the diameter of the holes is of order $\e^{\a} \ (\a>3)$, and the mutual distance between holes is proportional to $\e$ (subcritical case).
		
         This paper is organized as follows. In Section~\ref{INTRODUCTION}, we introduce the problem formulation, the definition of weak solutions and the main results. Section~\ref{sec:unifbds} is devoted to establishing uniform bounds on the velocity, density and temperature. An important compactness result is also introduced there as a crucial ingredient in the homogenization process of continuity equation. In Sections \ref{Momentum equations in homogenized domains}, we derive the extended momentum equations in the homogenized domain. Based on the convergence of the nonlinear convective term and the strong convergence result for velocity established in Section \ref{Convergence of the nonlinear convective term}, we derive the limit heat equation in Section \ref{Limit heat equation} and prove the strong convergence of temperature in Section \ref{Strong convergence of temperature}. Finally, we obtain the limit momentum equations in Section \ref{Limit momentum equations}.
	\paragraph{{\bf Notations.}} Throughout this paper, $L^p$ and $W^{k,p}$ represent Lebesgue and Sobolev spaces, respectively. $W^{1,q}_0$ denotes the subspace of $W^{1,q}$ with zero trace.   For a domain $D \subset \R^3$, we write $\|\cdot\|_{L^p L^q}=\|\cdot\|_{L^p(0,T;L^q(D))}$ and $\|\cdot\|_{L^p W^{k,q}}=\|\cdot\|_{L^p(0,T;W^{k,q}(D))}$. For matrices $A,B \in \R^{3 \times 3}$, the inner product is defined by $A:B = \sum_{i, j=1}^3 A_{ij} B_{ij}$. We use ~$\widetilde{\cdot}$ ~to denote the zero-extension for velocity. For density, we use ~$\widetilde{\cdot}$ ~to denote a positive constant extension. Moreover, we write $L^{q}_0(D)$ to denote the subspace of $L^q(D)$ consisting of functions with zero mean:
	\be
	L^q_0(D):=\left\{f\in L^q(D): \, \int_{D} f\,\dx=0\right\}. \nn
	\ee
		
\subsection{Problem formulation}\label{Problem formulation}
In this paper, we focus on three-dimensional case and take the diameter of holes in \eqref{sigma}
\be\label{size of holes}
a_{\e}=\e^{\a} \quad\mbox{with}\quad \a>3,
\ee
which implies $\s_{\e}=\e^{\frac{3-\a}{2}}$. We consider a bounded domain $\O \subset \R^3$ of class $C^{2,\b}$ with $0<\b<1$, and introduce a family of $\e$-dependent perforated domains $\{ \O_\e \}_{\e > 0}$ as
	\begin{align}\label{defi-holes-1}
	\O_\e = \O \setminus \bigcup_{k\in K_\e} T_{\e, k},\ && K_\e:=\{\e k\in \mathbb{Z}^3|\ \overline{Q}_{k, \e} \subset \O \}, && Q_{k, \e} := \e \Big( -\frac{1}{2},\frac{1}{2} \Big)^3 + \e k,\ \ k\in \mathbb{Z}^3,
	\end{align}
	where the sets $T_{\e,k}$ represent holes or obstacles. Suppose we have the following property concerning the distribution of the holes:
	\be\label{defi-holes-2}
	T_{\e,k} := x_{\e,k} + a_\e T_0 \subset\subset B(x_{\e,k}, b_0a_{\e} ) \subset\subset Q_{k, \e} \subset \O, 
	\ee
	with
	\be\label{defi-holes-3}
	|x_{\e, k} - x_{\e, l}| \geq 2\e \ \mbox{ for any } \ k \neq l \in K_\e. 
	\ee
	Here, $b_0$ is a positive constant independent of $\e$, $ T_0 \subset \R^3$ is a bounded, simply connected $C^{2,\b}$ domain contained in $Q_{0,1}$, and $B(x,r)$ denotes the open ball centered at $x$ with radius $r>0$. Without loss of generality, we assume that $T_0 \subset B(0, \frac12)$. The diameter of each $T_{\e,k}$ is of order $\mathcal{O}(a_\e)$. Meanwhile, the mutual distance between the holes is $\mathcal{O}(\e)$. The number of holes contained in $\O$ satisfies
	\be\label{defi-holes-4}
	|K_\e| \leq C(\e^{-3} |\Omega| + o(1)),\quad \mbox{for some} ~C>0~ \mbox{independent of} ~ \e.
	\ee
	
 For given $T>0$, we now consider the following non-homogeneous incompressible heat conduction equations in $(0,T) \times \O_{\e}$:
\ba\label{NHHC}
	\begin{cases}
		\d_{t}\vr_{\e}+\dive (\vr_{\e} \vu_{\e}) = 0, \ \ \dive \vu_{\e} = 0, &\mbox{in } (0,T) \times \O_{\e},\\
		\partial_t (\vr_\e \vu_\e) + \dive (\vr_\e \vu_\e \otimes \vu_\e) + \Grad P_\e = \dive \mathbb{S}(\Theta_\e, \Grad \vu_\e) - \Theta_\e \Grad F, &\mbox{in } (0,T) \times \O_{\e},\\
		\mathbb{S}(\Theta_\e, \Grad \vu_\e) = \mu(\Theta_\e) \left( \Grad \vu_\e + \Grad^{\textup{T}} \vu_\e \right), & \mbox{in } (0,T) \times \O_{\e},\\
		- \dive \left( \kappa_\e \Grad \Theta_\e \right) = \Grad F \cdot \uu_\e,& \mbox{in } (0,T) \times \O_{\e},\\
		\vr_{\e}|_{t=0} = \vr_{\e}^{0},\quad (\vr_{\e}\vu_{\e})|_{t=0} = \vr_{\e}^{0}\vu_{\e}^{0} & \mbox{in } \O_\e.
	\end{cases}
\ea
Here the mass density $\vr_\e = \vr_\e(t,x)$, the velocity $\vu_\e = \vu_\e(t,x)$ and the temperature $\Theta_\e = \Theta_\e(t,x)$ satisfy a variant of the Navier-Stokes-Boussinesq system introduced by Chandrasekhar \cite{CHAN} (see also Ligni\`eres \cite{LIGN}). The viscosity coefficient $\mu(\cdot): \mathbb R\to \mathbb R^+$ is a Lipschitz continuous function (related to temperature $\Theta_\e$), which satisfies
\be\label{mu(theta)-low-high}
0< \underline\mu \leq \mu(\cdot) \leq \bar\mu <\infty,
\ee
where $\underline\mu$ and $\bar\mu$ are positive constants. The last equation $\eqref{NHHC}_4$ can be viewed as a quasi-static approximation of the conventional heat equation in the high P\' eclet number:
\ba\nn
\partial_t (\vr_\e \Theta_\e) + \dive (\vr_\e \Theta_\e \vu_\e) - \dive \left( \kappa \Grad \Theta_\e \right) = \Grad F \cdot \vu_\e,
\ea
where $F$ denotes the gravitational potential. We refer to \cite[Chapter 4, Section 4.3]{FENO6} for a rigorous derivation of system \eqref{NHHC} in the spatially homogeneous case $\vr \equiv 1$.

The fluid is contained in a bounded domain $\Omega_\e \subset \R^3$, and the velocity satisfies the no-slip boundary condition
\begin{equation} \label{boundary condition for u}
\vu_\e|_{\partial \Omega_\e} = 0.
\end{equation}
We extend the velocity $\vu_\e$ by
\ba\nn
\widetilde\vu_\e(t,x)= \left\{ \begin{array}{l}\vu_\e(t,x), \ x\in \Omega_\e, \\ \\
\mathbf 0, \ \ \ \ \ \ \ \ \ x\in \R^3\backslash \Omega_\e. \end{array} \right.
\ea

Since the heat conductivity coefficient is a positive constant in $\Omega_\e$, we suppose $\eqref{NHHC}_4$ is satisfied in $\Omega$, with heat conductivity coefficient $\kappa_\e$ satisfying
\ba\label{defi of k}
\kappa_\e (x) = \left\{ \begin{array}{l} \kappa_f > 0, \ \mbox{if}\ x \in \Omega_\e, \\ \\
\kappa_s > 0, \ \mbox{if}\ x \in \Omega \setminus {\Omega}_\e, \end{array} \right.
\ea
where, in general, we allow $\kappa_f \ne \kappa_s$.  For definiteness, we impose the homogeneous Dirichlet boundary condition for the temperature,
\begin{equation} \label{boundary condition for temp}
\Theta_\e |_{\partial \Omega} = 0.
\end{equation}

To establish the connection of the equations in $\O_\e$ and $\O$, firstly, we will assume that the positive constant extension of $\vr_\e^0$ and the zero extension of $\vu_\e^0$ satisfy
 \begin{eqnarray}\label{initial-rho-u}
\left\{
\begin{aligned}
&\widetilde\vr_{\e}^{0}\in L^{\infty}(\mathbb R^{3}),\quad \widetilde\vr_{\e}^{0}(x)=\overline\vr>0 \ \ \textup{in } \mathbb R^{3}\backslash{\O}_{\e},\\
&0<\underline\vr\leq\widetilde\vr_{\e}^{0}(x)\leq\overline\vr,\quad \widetilde\vr_{\e}^{0}\rightarrow \vr_{0} \ \ \textup{strongly in }L^{\infty}({\Omega}),\\
&\dive \widetilde\vu_{\e}^{0} = 0 \ \ \textup{in }\O_{\e},\quad \widetilde\vu_{\e}^{0} = \mathbf{0} \ \ \textup{in }\mathbb R^{3}\backslash{\O}_{\e},\quad \widetilde\vu_{\e}^{0}\rightarrow \vu_{0} \ \ \textup{strongly in }L^{2}(\Omega; \mathbb R^{3}),
\end{aligned}
\right .
\end{eqnarray}
where $\underline\vr$ and $\overline\vr$ are both positive constants. Suppose that the gravitational potential $F$ satisfies:
\ba\label{assumption of F}
\Grad F \in L^{\infty}(\Omega; \R^3).
\ea
Without loss of generality, we assume that the density $\vr_\e$ is extended from $\O_\e$ to $\R^{3}$ by positive constant $\overline\vr$, i.e. for $x \in \mathbb R^3 \setminus {\Omega}_\e$, the constant extension for density is as follows
\ba\nn
\widetilde\vr_{\e}(t,x)=\widetilde\vr_{\e}^{0}(x)=\Ov{\vr}>0.
\ea
Combining the extensions of  $(\vr_\e, \vu_\e)$ given above, we may therefore assume that the continuity equation in \eqref{NHHC} is satisfied in the whole space $\R^3$.

\subsection{Weak solutions}\label{Weak solutions}
Let us introduce the definition of finite energy weak solutions to the non-homogeneous incompressible heat conduction equations \eqref{NHHC}--\eqref{boundary condition for temp}.
\begin{defi}\label{fews}
Let
\ba\label{init}		
\vr_\e(0, \cdot) = \vr_\e^0, \ \vu_\e(0, \cdot) = \vu_\e^0.
\ea		
A triple $(\vr_{\e},\vu_{\e},\Theta_\e)$ is called a {\bf\emph{finite energy weak solution}} to system \eqref{NHHC}--\eqref{boundary condition for temp} provided:	
\begin{itemize}
				
\item the $(\vr_{\e},\vu_{\e},\Theta_\e)$ satisfies:
\ba \label{original regularity}
&\vr_\e \in C([0,T]; L^1(\Omega_\e)), \quad 0 < \underline{\vr} \leq \vr_\e \leq \Ov{\vr} \ \ \mbox{a.a. in}\ (0,T) \times \Omega_\e,\\
&\vu_\e \in L^\infty(0,T; L^2(\Omega_\e; \R^3)) \cap L^2(0,T; W^{1,2}_0(\Omega_\e; \R^3)),\\
&\Theta_\e \in L^\infty(0,T; W^{1,2}_0(\Omega)); 
\ea

\item for any $\psi\in C_{c}^{\infty}([0, T) \times \R^3)$, 
\be\label{continuity eq}
\int_{0}^{T}\int_{\R^3}\widetilde\vr_{\e}(\d_{t}\psi+\widetilde\vu_{\e}\cdot\Grad\psi) \dx \dt = -\int_{\R^3}\widetilde\vr_{\e}^{0}\psi(0,x) \dx;
\ee
\item for a.a. $\tau \in(0,T)$, $\dive \vu_\e = 0 ~ \mbox{in} ~ \Omega_\e;$	
			
\item for any $\bvp \in C_{c}^{\infty}([0,T) \times \O_{\e}; \R^{3})$, $\dive\bvp=0$,
\ba\label{momentum eq}
&\int_{0}^{T}\int_{\O_{\e}} \big[ \vr_{\e}\vu_{\e}\cdot\d_{t} \bvp+(\vr_{\e} \vu_{\e} \otimes \vu_{\e}):\Grad \bvp \big] \dx \dt -\int_0^T \intOe{ \mathbb{S}(\Theta_\e, \Grad \vu_\e ) : \Grad \bvp } \dt\\
&\quad= \int_0^T \intOe{ \Theta_\e \Grad F \cdot \bvp } \dt- \intOe{ \vr_\e^0 \vu_\e^0 \cdot \bvp(0, \cdot) };
\ea
			
\item for any $\phi \in C^{\infty}_c ((0,T)\times\Omega)$,
\ba\label{temp eq}
\int_0^T\intO{ \kappa_\e \Grad \Theta_\e  \cdot \Grad \phi }\dt = \int_0^T\intO{ \widetilde\vu_\e  \cdot \Grad F \phi }\dt;
\ea

\item for a.a. $\tau \in (0,T)$, there holds the energy inequality
\ba\label{energy inequality}
&\intOe{ \frac{1}{2} \vr_\e |\vu_\e|^2 (\tau, \cdot) }+ \int_0^\tau \intO{ \kappa_\e |\Grad \Theta_\e|^2 } \dt + \int_0^\tau \intOe{ \frac{\mu(\Theta_\e)}{2} |\Grad \vu_\e + \Grad^{\textup{T}} \vu_\e |^2 } \dt\\
&\quad\leq \intOe{ \frac{1}{2} \vr_\e^0 |\vu_\e^0 |^2 }.
\ea
\end{itemize}
\end{defi}	
Thanks to DiPerna-Lions theory \cite{DL} and the existing regularity of $\vr_\e$, $\vu_\e$ stated in (\ref{original regularity}), the weak formulation \eqref{continuity eq} for continuity equation admits the following renormalized variant
\begin{equation} \label{renormalized continuity eq}
\int_0^T \int_{\mathbb R^3}{ \left[ b(\widetilde\vr_\e ) \partial_t \varphi + b(\widetilde\vr_\e ) \widetilde\vu_\e \cdot \Grad \varphi \right] } \dx\dt = - \int_{\mathbb R^3}{ b(\widetilde\vr_\e^0) \varphi(0, \cdot)}\dx
\end{equation}
for any $\varphi \in C^{\infty}_c([0,T) \times \R^3)$ and any $b \in C((0,\infty))$.

Notice that, for any fixed $\e>0$, following the well-known argument in Lions' book \cite{Lions-Incom}, the existence of renormalized weak solutions can be derived.
	
\subsection{Main results} 
 
Now, we are in a position to state the main results of this paper.
\begin{theorem} \label{Main theorem}
Let $\{ \Omega_\e\}_{\e > 0}$ be a family of perforated domains specified in Section \ref{Problem formulation}. Let $( \vr_\e, \vu_\e, \Theta_\e)$ be a finite energy weak solution to the problem \eqref{NHHC}--\eqref{boundary condition for temp}. Under the assumptions of initial data and external force in \eqref{initial-rho-u}--\eqref{assumption of F}, up to a subsequence,  $(\widetilde\vr_\e, \widetilde\vu_\e, \Theta_\e)$ satisfy
\ba\label{convergence}
&\widetilde\vr_\e \to \vr ~\mbox{in}\ C([0,T]; L^p(\Omega)) \ \ \mbox{for any}\ 1 \leq p< \infty, \ \ 0 < \underline{\vr} \leq \widetilde\vr_{\e}(t, x) \leq \Ov{\vr}; \\
&\widetilde\vu_\e \to \vu \ \mbox{strongly in}\ L^2((0,T) \times \Omega; \R^3) \ \mbox{and weakly in}\ L^2(0,T; W^{1,2}_0(\Omega; \R^3));\\
&\Theta_\e \to \Theta \ \mbox{strongly in}\ L^2(0,T; W^{1,2}_0(\Omega)),
\ea
where $(\vr, \vu, \Theta)$ is a weak solution to the problem
\begin{align}\label{Darcy}
\begin{cases}
 \d_t \vr + \dive(\vr \uu) = 0,\quad \dive \vu=0, & \mbox{in } (0,T) \times \O,\\
 \partial_t (\vr \vu) + \dive (\vr \vu \otimes \vu) + \Grad P = \dive \mathbb{S}(\Theta, \Grad \vu) + \Theta \Grad F, & \mbox{in } (0,T) \times \O,\\
 \mathbb{S}(\Theta, \Grad \vu) = \mu(\Theta) \left( \Grad \vu + \Grad^{\textup{T}} \vu \right), & \mbox{in } (0,T) \times \O,\\
 -\dive \left( \kappa_f \Grad \Theta \right) = \Grad F \cdot \vu,& \mbox{in } (0,T) \times \O,
\end{cases}
\end{align}
 with boundary conditions
 \ba\label{boundary conditions for limit eq}
\vu |_{\partial \Omega} = \mathbf{0},~~\Theta |_{\partial \Omega} = 0,
\ea
and initial conditions
 \ba\label{initial conditions for limit eq}
\vr|_{t=0} = \vr_{0},~~ (\vr\vu)|_{t=0} = \vr_{0}\vu_{0}~~\mbox{in } ~\O.
\ea 
\end{theorem}

\begin{remark}
The definition of the weak solution to system \eqref{Darcy}-\eqref{initial conditions for limit eq}  is similar to the weak solution $(\vr_{\e},\vu_{\e},\Theta_\e)$ given in Definition \ref{Weak solutions}.
\end{remark}

The rest of the paper is devoted to the proof of above theorem. Throughout the sequel, $C$ denotes a generic positive constant independent of $\e$, which may change from line to line.

\section{Uniform bounds}\label{sec:unifbds}
Firstly, we will consider the uniform bounds on $\widetilde\vr_\e$. Employing initial condition hypothesis in \eqref{initial-rho-u}, we can take 
\ba\nn
b(\widetilde\vr_{\e})=[\widetilde\vr_{\e}-\overline\vr]^{+},\quad b(\vr_{\e})=-[\widetilde\vr_{\e}-\underline\vr]^{-}
\ea
as test functions in the renormalized equation \eqref{renormalized continuity eq} to deduce
\ba \label{estimate for rho}
0 < \underline{\vr} \leq \widetilde\vr_\e(t,x) \leq \Ov{\vr} \ \ \mbox{for a.a.}\ (t,x)\in (0,T) \times \O,
\ea
uniformly in $\e \to 0$. Combining the lower bounds in \eqref{mu(theta)-low-high} for $\mu(\Theta_\e)$ and in \eqref{estimate for rho} for $\widetilde\vr_\e$, and using the energy inequality \eqref{energy inequality}, we deduce that, for a.a. $\tau \in (0,T)$, 
\ba\label{energy inequality estimate}
& \underline{\vr} /2\intOe{|\vu_\e|^2(\tau, \cdot)}+ \min \{ \kappa_s, \kappa_f \}\int_0^\tau \intO{ |\Grad \Theta_\e|^2 } \dt + \underline\mu/2\int_0^\tau \intOe{|\Grad \vu_\e + \Grad^{\textup{T}} \vu_\e |^2 } \dt\\
&\quad\leq\intOe{ \frac{1}{2} \vr_\e |\vu_\e|^2(\tau, \cdot)}+ \int_0^\tau \intO{ \kappa_\e |\Grad \Theta_\e|^2 } \dt + \int_0^\tau \intOe{ \frac{\mu(\Theta_\e)}{2} |\Grad \vu_\e + \Grad^{\textup{T}} \vu_\e |^2 } \dt\\
&\quad\leq \intOe{ \frac{1}{2} \vr_\e^0 |\vu_\e^0|^2}.
\ea
 Applying Korn inequality gives
\begin{equation} \label{eatimates for u and th}
 \left\|\vu_\e \right\|_{L^\infty(0,T; L^2(\Omega_\e))} +\|\sqrt{\vr_\e}\uu_\e \|_{L^\infty(0,T;L^2(\Omega_\varepsilon))}+ \int_0^T \|\Grad \Theta_\e \|^2_{L^2(\Omega)}\dt+\int_0^T \|\Grad \vu_\e \|^2_{L^2(\Omega_\e)}\dt \leq C,
\end{equation}
where the constant in \eqref{eatimates for u and th} depends on $\underline{\vr}$, $\kappa_s, \kappa_f$, $\underline\mu$ and initial data $(\vr_\e^0,\vu_\e^0)$.  Thus the Zero-extension $\widetilde\vu_\e$ satisfies
\ba\label{eatimates for u}
& \|\widetilde\uu_\e \|_{L^\infty(0,T;L^2(\Omega))} \leq C,\quad\|\Grad\widetilde\uu_\e \|_{L^2(0,T;L^2(\Omega))} \leq C.
\ea
Up to a subsequence, we obtain
\ba \label{Covergence of u}
&\widetilde\vu_\e \to \vu \ \mbox{weakly-(*) in}\ L^\infty(0,T; L^2(\Omega; \R^3)) \ \mbox{and weakly in}\ L^2(0,T; W^{1,2}_0(\Omega;\R^3)).
\ea
Moreover, by means of Poincar\'e inequality, \eqref{eatimates for u and th} gives
\ba\label{uniform bound of theta}
\|\Theta_\e\|_{L^2(0,T; W_0^{1,2}(\Omega))} \leq C\|\Grad\Theta_\e\|_{L^2(0,T; L^2(\Omega))} \leq C,
\ea
 from which we can derive 
\ba\label{weak convergence of th}
\Theta_\e \to \Theta\ \mbox{weakly in}\ L^2(0,T; W^{1,2}_0(\Omega)). 
\ea

Note that $\widetilde\vu_\e \equiv \mathbf{0}$ and $\widetilde\vr_\e\equiv\Ov{\vr}>0$ outside $\Omega_\e$. By \eqref{estimate for rho}, we have $0\leq\widetilde\vr_\e\leq C$ a.a. in $(0,T)\times\Omega$. The convergence of $\widetilde\vr_\e^0$ also holds in $L^p(\Omega)$ for all $1\leq p\leq\infty$. Based on the above analysis, we have properties as follows
\ba\label{assume-compact}
&0\leq\widetilde\vr_\e\leq C \ \mbox{a.a. on}\ (0,T)\times\Omega;\\
&\dive \widetilde\uu_\e =0\ \mbox{a.a. on}\ (0,T)\times\Omega,\ \ \|\widetilde \uu_\varepsilon \|_{L^2 (0,T;W_0^{1,2}(\Omega))} \leq C;\\
& \partial_t \widetilde\vr_\e +\dive (\widetilde\vr_\e\widetilde\uu_\e) = 0,\ \ \mbox{in}\ \mathcal{D}'((0,T)\times\RR^3 );\\
&\widetilde\vr_\e^0\to\vr^0\ \mbox{strongly in}\ L^1(\Omega),
\ \ \widetilde\uu_\e\to\uu\  \mbox{weakly in}\ L^2 (0,T;W_0^{1,2}(\Omega;\mathbb R^3)).
\ea
Now we recall the following compactness result (see Theorem 2.4 in \cite{Lions-Incom}): 
\begin{proposition}\label{compact-density}
If \eqref{assume-compact} holds, then $\widetilde\vr_\e$ converges strongly in $C([0,T];L^p(\Omega))$ for all $1\leq p<\infty$ to the unique solution $\vr$ bounded on $(0,T)\times\Omega$, of
\ba\nn
\begin{cases}\nn
\partial_t \vr +\dive (\vr\uu) = 0,&\mbox{in}\ \mathcal{D}'((0,T)\times\R^3),\\
\vr\in C([0,T];L^1(\Omega)), & \vr(0)=\vr^0 \ \mbox{a.a. in} \ \Omega.\\
\end{cases}
\ea
\end{proposition}
Using Proposition \ref{compact-density}, we obtain the compactness result
\ba\label{Covergence of rho}
\widetilde\vr_\e \to \vr \ \mbox{strongly in} \ C([0,T];L^p(\Omega))\ \mbox{for all}\ 1\leq p<\infty,
\ea
with $\vr\in C([0,T];L^p(\Omega)), \ 1\leq p<\infty, \ \vr(0)=\vr^0 \ \mbox{a.a. in} \ \Omega.$ 

Then, using \eqref{Covergence of u}, \eqref{weak convergence of th} and \eqref{Covergence of rho}, we have
\ba \label{Covergence of solu}
&\widetilde\vr_\e \to \vr \ \mbox{in}\ C([0,T]; L^p(\Omega)) \ \mbox{for any}\ 1 \leq p< \infty,\\
&\widetilde\vu_\e \to \vu \ \mbox{weakly-(*) in}\ L^\infty(0,T; L^2(\Omega; \R^3)) \ \mbox{and weakly in}\ L^2(0,T; W^{1,2}_0(\Omega;\R^3)),\\
&\Theta_\e \to \Theta\ \mbox{weakly in}\ L^2(0,T; W^{1,2}_0(\Omega)). 
\ea
The standard Aubin-Lions argument immediately implies that $(\vr$, $\vu)$ satisfies \eqref{continuity eq}. Meanwhile, by DiPerna-Lions theory \cite{DL}, the renormalized equation \eqref{renormalized continuity eq} holds for $(\vr$, $\vu)$. Furthermore, we apply \eqref{estimate for rho} and \eqref{eatimates for u} to find
\ba\label{est-rho u}
&\|\widetilde\vr_\e\widetilde \uu_\e \|_{L^2 L^6}\leq \|\widetilde\vr_\e \|_{L^\infty L^\infty} \|\widetilde \uu_\e \|_{L^2 L^6}\leq C,
~~\|\widetilde\vr_\e\widetilde \uu_\e \|_{L^\infty L^2}\leq\|\widetilde\vr_\e\|_{L^\infty L^\infty} \|\widetilde \uu_\e\|_{L^\infty L^2}\leq C.
\ea

\section{Asymptotic limit}\label{Equation in homogeneous domain}
In this section, we will show the limit equations of the non-homogeneous heat conduction equations in the homogenized domain $\O$, together with the corresponding strong convergence results. 

 Actually, Proposition \ref{compact-density} shows the limit continuity equation in Theorem \ref{Main theorem} directly, i.e.,
\ba\nn
\partial_t \vr +\dive (\vr\uu)= 0,
\ea
which means the limit continuity equation remains unchanged.  At the same time, the divergence free condition $\dive \vu=0$ follows from $\dive \widetilde\vu=0$. Thus we focus on the momentum equations and heat equation in homogenized domain.

\subsection{Momentum equations in homogenized domain}\label{Momentum equations in homogenized domains}
To deal with the momentum equations in $\Omega$, one may introduce a family of functions ${\{ {g_\varepsilon}\} _{\varepsilon>0}}$ that vanish on the holes and converge to $1$ in an appropriate Sobolev space $W^{1,q}(\Omega)$. For any $\bvp\in C^\infty_c([0,T)\times\Omega ;\mathbb{R}^3)$ with $\dive \bvp =0$, we decompose it into
\be\nn
\bvp = {g_\varepsilon }\bvp + (1-{g_\varepsilon})\bvp.
\ee
For the terms involving $(1-g_\varepsilon)\bvp$, we need to show that they converge to zero. However, this decomposition destroys the divergence-free property of $\bvp$: $\dive (g_{\e} \bvp) \neq 0$. Thus $g_{\e} \bvp$ is not an admissible test function in \eqref{momentum eq}. To solve this problem, we recall a Bogovskii type operator in perforated domain $\Omega_{\e}$ (see Theorem 2.3 in \cite{DFL} and Proposition 2.2 in \cite{Lu-Schwarz18}):
\begin{lemma}\label{lem-div}
Let $\{\Omega_\e\}_{\e>0}$ be a family of perforated domains defined through \eqref{defi-holes-1}--\eqref{defi-holes-4} with $a_{\e} = \e^{\alpha}$, $\alpha\geq 1$. Then there exists a linear operator
\ba\nn
\calB_\e : L_0^{q}(\Omega_\e) \to W_0^{1,q}(\Omega_\e; \R^3),\ 1 < q < \infty,
\ea
such that for any $f\in L_0^{q}(\Omega_\e)$,
\ba\label{pro-div1}
\dive \calB_\e(f) =f \ \mbox{in} \ \Omega_\e,~~\|\calB_\e(f)\|_{W_0^{1,q}(\Omega_\e; \R^3)}\leq C\, \big(1+\e^{\frac{(3-q)\a-3}{q}}\big)\|f\|_{L^q(\Omega_\e)},
\ea
for some constant $C$ independent of $\e$.
\end{lemma}

In order to derive the limit momentum equations in $[0,T)\times\Omega$, we give the following proposition to describe the extended momentum equations with remainder term.
\begin{proposition}\label{moment-equa}
Under the assumptions in Theorem \ref{Main theorem},  $(\widetilde\vr_\e,\widetilde \uu_\varepsilon, \Theta_\e)$ satisfies the following equation:
\ba\label{momentum eq-1}
&\int_{0}^{T}\int_{\O} \big[ \widetilde\vr_{\e}\widetilde\vu_{\e}\cdot\d_{t} \bvp+(\widetilde\vr_{\e} \widetilde\vu_{\e} \otimes \widetilde\vu_{\e}):\Grad \bvp \big] \dx\dt -\int_0^T \intO{ \mathbb{S}(\Theta_\e, \Grad \widetilde\vu_\e) : \Grad \bvp }\dt \\
&\quad= \int_0^T \intO{\Theta_\e \Grad F \cdot \bvp } \dt
- \intO{ \widetilde\vr_\e^0 \widetilde\vu_\e^0 \cdot \bvp(0, \cdot) }+\l\GG_\varepsilon, \bvp \r,
\ea
where $\bvp\in C^\infty_c([0,T)\times\Omega ;\mathbb{R}^3)$ with $\dive \bvp =0$. Moreover,
\begin{equation}\label{est-F-3D}
|\l\GG_\varepsilon, \bvp \r| \leq C\varepsilon ^\sigma \big(\|\partial _t \bvp \|_{L^{\frac 43} L^{2}}+\|\Grad \bvp\|_{L^{4} L^{r_1}}+\|\bvp (0,x)\|_{L^{r_2}} \big), \ \forall \, \bvp\in C_{c}^{\infty}([0,T)\times\Omega;\mathbb{R}^3).
\end{equation}
Here $\sigma:=((3-q)\alpha-3)/q >0$ for some $q>2$ close to $2$, and $ 2<r_1<3$ is given by ${1}/{r_{1}} = {5}/{6} -{1}/{q}$, and $1<r_2<\infty$ satisfies $1/2=1/q+1/{r_2}$.
\end{proposition}
\begin{proof} 
Let $\bvp \in C_c^\infty ([0,T)\times\Omega;{\mathbb{R}^3})$ with $\dive\bvp = 0$. Combining the distribution of holes in \eqref{defi-holes-1}--\eqref{defi-holes-4}, there exist cut-off functions $\{ g_\varepsilon \}_{\varepsilon>0} \subset C^\infty(\R^{3})$ such that $0\leq g_{\e} \leq 1$ and
\begin{equation}\label{defi-g}
 g_\varepsilon= 0 \ \ \mbox{on} \ \ \bigcup_{k \in K_\varepsilon} T_{\e,k},\quad g_\varepsilon = 1 \ \ \mbox{on} \ \ (\bigcup_{k \in K_\varepsilon} B(x_{\e,k},\delta_0 \varepsilon^\alpha))^c,\quad |\Grad g_{\e}| \leq C \e^{-\alpha}.
\end{equation}
Then for each $1\leq q \leq \infty$, there hold
\begin{equation}\label{est-g}
\|g_\varepsilon- 1\|_{L^q(\R^{3})} \leq C \varepsilon^{\frac{3\alpha-3}{q}},\quad
\|\Grad g_\varepsilon\|_{L^q(\R^{3})} \leq C\varepsilon ^{\frac{3\alpha-3}{q}-\alpha}.
\end{equation}
Now we estimate
\ba\nn
I^\varepsilon :&=\int_{0}^{T}\int_{\O} \big[ \widetilde\vr_{\e}\widetilde\vu_{\e}\cdot\d_{t} \bvp+(\widetilde\vr_{\e} \widetilde\vu_{\e} \otimes \widetilde\vu_{\e}):\Grad \bvp \big] \dx\dt-\int_0^T \intO{ \mathbb{S}(\Theta_\e, \Grad \widetilde\vu_\e) : \Grad \bvp }\dt\\
& \quad- \int_0^T \intO{ \Theta_\e \Grad F \cdot \bvp } \dt
+\intO{ \widetilde\vr_\e^0 \widetilde\vu_\e^0 \cdot \bvp(0, \cdot) }.
\ea
Firstly, we decompose $\bvp$ into $\bvp = {g_\varepsilon }\bvp + (1-{g_\varepsilon})\bvp$ and this allows us to rewrite 
\ba\nn
I^\varepsilon :&=\int_{0}^{T}\int_{\O_{\e}} \big[ \widetilde\vr_{\e}\widetilde\vu_{\e}\cdot\d_{t} (g_\e\bvp)+(\widetilde\vr_{\e} \widetilde\vu_{\e} \otimes \widetilde\vu_{\e}):\Grad (g_\e\bvp) \big] \dx\dt-\int_0^T \intOe{ \mathbb{S}(\Theta_\e, \Grad \widetilde\vu_\e) : \Grad (g_\e\bvp) }\dt\\
& \quad- \int_0^T \intOe{ \Theta_\e \Grad F \cdot (g_\e\bvp) } \dt+\intOe{ \widetilde\vr_\e^0 \widetilde\vu_\e^0 \cdot (g_\e\bvp)(0, \cdot) }+\sum_{i=1}^{5} {I_i},
\ea
with
\ba\nn
I_1 & = \int_0^T \int_\Omega \widetilde\vr_{\e}\widetilde\vu_{\e}\cdot(1-g_\e)\d_{t} \bvp\dx \dt, \\
I_2 & = \int_0^T \int_\Omega \widetilde\vr_\e\widetilde \uu_\e \otimes \widetilde \uu_\e :(1-g_\e)\Grad \bvp-\widetilde\vr_\e\widetilde\uu_\e \otimes \widetilde \uu_\e:\Grad g_\e\otimes\bvp \dx\dt ,\\
I_3 &= -\int_0^T \int_\Omega \mathbb{S}(\Theta_\e, \Grad \widetilde\vu_\e) :(1-g_\varepsilon)\Grad \bvp
- \mathbb{S}(\Theta_\e, \Grad \widetilde\vu_\e):\Grad g_\varepsilon \otimes \bvp \dx \dt ,\\
I_4 &= \int_0^T \int_\Omega \Theta_\e \Grad F \cdot(g_\varepsilon-1)\bvp \dx\dt,\\
I_5 &= -\int_{\Omega} \widetilde{\vr}_\e^0\widetilde{\uu}_\e^0 \cdot (g_\varepsilon-1)\bvp (0,x)\dx.
\ea

Observing that
\ba\label{0-integrals}
\int_{\Omega_{\e}} \bvp \cdot \Grad g_{\e}\dx = \int_{\Omega_{\e}} \dive(\bvp g_{\e}) \dx = 0,
\ea
we apply Bogovskii operator in Lemma \ref{lem-div} to separate a divergence-free part from $g_{\e} \bvp$, that is,
\ba\nn
\bvp_{1} := g_\e\bvp - \bvp_{2},\quad \bvp_{2} := \calB_{\e} (\dive(\bvp g_{\e})) = \calB_{\e} (\bvp \cdot \Grad g_\e).
\ea
where $\bvp_{1} \in C_{c}^{\infty}([0,T); W^{1,2}_{0}(\Omega_{\e};\mathbb R^3))$ and $\dive \bvp_{1} = 0$. As a result, $\bvp_{1}$ is an an admissible test function in \eqref{momentum eq} and it follows from weak formulation \eqref{momentum eq} that
\ba\nn
I^\varepsilon&=\int_{0}^{T}\int_{\O_{\e}} \big[\vr_{\e}\vu_{\e}\cdot\d_{t} \bvp_{1} +(\vr_{\e} \vu_{\e} \otimes \vu_{\e}):\Grad \bvp_{1} \big] \dx\dt-\int_0^T \intOe{ \mathbb{S}(\Theta_\e, \Grad \vu_\e) : \Grad \bvp_{1} }\dt\\
&\quad - \int_0^T \intOe{\Theta_\e \Grad F \cdot \bvp_{1}} \dt+\intOe{\vr_\e^0 \vu_\e^0 \cdot \bvp_{1} (0, \cdot) }+\sum_{ i=1}^{5} {I_i} +\sum_{ i=6}^{10} {I_i}\\
&=\sum_{ i=1}^{5} {I_i} +\sum_{ i=6}^{10} {I_i},
\ea
where
\ba\nn
&I_6 =\int_0^T \int_{\Omega_\e}\vr_\e\uu_\e \partial_t \bvp_{2} \dx \dt,\quad &&
~~I_7 = \int_0^T \int_{\Omega_\e} \vr_\e\uu_\e \otimes \uu_\e : \Grad \bvp_{2} \dx \dt , \\
&I_8 =-\int_0^T \int_{\Omega_\e} \mathbb{S}(\Theta_\e, \Grad\vu_\e):\Grad \bvp_{2} \dx \dt ,\quad &&
~~I_9 = -\int_0^T \int_{\Omega_{\e}} \Theta_\e \Grad F \cdot \bvp_{2}\dx\dt,\\
&I_{10} =\int_{\Omega _\e} \vr_\e^0\uu_\e^0 \cdot \bvp_{2} (0,x)\dx.
\ea

Now we estimate $I_{i}$ term by term. Since $\alpha>3$, there exists $q\in(2, 3)$ close to $2$ such that
\be\nn
\sigma:=\frac{(3 - q)\alpha-3}{q} >0.
\ee
 By \eqref{estimate for rho}, \eqref{eatimates for u} and \eqref{est-g}, together with interpolation inequality and Sobolev embedding theorem, we have
\ba\nn
|I_1| \leq C \|\widetilde\vr_\e \widetilde \uu_\e \|_{L^4 L^3}\|1- g_\e \|_{L^6}\|\partial_t\bvp\|_{L^{\frac43} L^2}\leq C \e^\sigma\|\partial _t\bvp \|_{L^{\frac43} L^2} .
\ea
Let $q^*$ be Sobolev conjugate component of $q$ with $\frac{1}{q^{*}} = \frac{1}{q} - \frac{1}{3}$. Clearly $6<q^{*}<\infty$. By \eqref{eatimates for u} and \eqref{est-rho u}, using interpolation inequality and Sobolev embedding theorem again, we obtain 
\ba
|I_2| \leq C\|\widetilde\vr_\e \widetilde \uu_\e \otimes \widetilde \uu_\e \|_{L^{\frac{4}{3}} L^2}\big(
\|g_\e- 1\|_{L^{q^{*}}}\|\Grad \bvp \|_{L^4L^{r_{1}}} + \|\Grad g_\e\|_{L^q}\|\bvp \|_{L^{4} L^{r_2}}\big) \leq C \e ^\sigma\|\Grad \bvp \|_{L^{4} L^{r_1}} ,
\nonumber\ea
with
\ba\label{r1}
\frac{1}{r_{1}} = \frac{1}{2} - \frac{1}{q^{*}} = \frac{5}{6} - \frac{1}{q}>\frac{1}{3}, \quad \frac{1}{r_{2}} = \frac{1}{2} - \frac{1}{q} = \frac{1}{r_{1}^{*}}.
\ea
Similarly, utilizing \eqref{mu(theta)-low-high}, \eqref{assumption of F}, \eqref{eatimates for u} and \eqref{est-g}, we have 
\ba\nn
|I_3| &\leq C\|\mathbb{S}(\Theta_\e, \Grad \widetilde\vu_\e) \|_{L^2 L^2} \big(\|g_\varepsilon - 1\|_{L^{q^*}}\|\Grad \bvp\|_{L^2 L^{r_1}} +\|\Grad g_\varepsilon\|_{L^q}\|\bvp\|_{L^2 L^{r_2}}\big)\\
&\leq C\|\mu(\Theta_\e) \|_{L^\infty L^\infty} \|\Grad \widetilde\vu_\e\|_{L^2 L^2} \big(\|g_\varepsilon - 1\|_{L^{q^*}}\|\Grad \bvp\|_{L^2 L^{r_1}} +\|\Grad g_\varepsilon\|_{L^q}\|\bvp\|_{L^2 L^{r_2}}\big)\\
&\leq C\varepsilon ^\sigma\|\Grad \bvp\|_{L^2 L^{r_1}},
\ea
and
\ba\nn
|I_4| \leq C\|\Theta_\e \|_{L^\infty L^2} \|\Grad F\|_{ L^\infty}\|g_\e- 1\|_{L^q}\|\bvp \|_{L^2 L^{r_2}}
\leq C \e ^\sigma\|\Grad\bvp \|_{L^2 L^{r_1}}.
\ea
It follows from \eqref{initial-rho-u} that
\ba\nn
|I_5| \leq C\|\widetilde\vr^0_\e \|_{L^\infty} \|\widetilde{\uu}_\e^0\|_{L^2}\|g_\e- 1\|_{L^q}\|\bvp(0,x)\|_{L^{r_2}}
\leq C \e ^\sigma\|\bvp(0,x)\|_{L^{r_2}}.
\ea

Next, we will repeatedly use properties of the Bogovskii operator $\calB_{\e}$ in Lemma \ref{lem-div} to estimate $I_{i}$ $(i=6,7,8,9,10)$. Since Bogovskii operator $\calB_{\e}$ only acts on spatial variable, by \eqref{0-integrals}, we have
 \ba\nn
 \d_{t} \bvp_{2} = \d_{t} \calB_{\e}(\bvp \cdot \Grad g_{\e}) = \calB_{\e}(\d_{t}\bvp \cdot \Grad g_{\e}).
 \ea
 Since $q>2$ close to $2$, there exist $r_{3} > \frac 32$ close to $\frac 32$ and $r_{4}>1$ close to $1$ such that
$$
 \frac{1}{r_{3}} = \frac{1}{r_{4}} - \frac 13, \quad \frac{1}{q} + \frac{1}{2} = \frac{1}{r_{4}}.
$$
Actually, \eqref{0-integrals} shows that $\d_{t}\bvp \cdot \Grad g_{\e}$ has zero integral mean. Then using {\eqref{pro-div1}}, \eqref{est-g} and Sobolev embedding, we have
 \ba\nn
|I_6|& \leq\|\vr_\e\uu_\e \|_{L^4 L^3} \|\d_{t} \bvp_{2} \|_{L^{\frac{4}{3}} L^{r_{3}}} 
\leq C \|\d_{t} \bvp_{2} \|_{L^{\frac{4}{3}} W^{1, r_{4}}} \\
& \leq C \Big(1+\e^{\frac{(3-r_{4})\a-3}{r_{4}}}\Big)\|\d_{t}\bvp \cdot\Grad g_{\e}\|_{L^{\frac 43} L^{r_{4}}} \\
& \leq C\|\Grad g_{\e}\|_{L^{q}}\|\d_{t} \bvp \|_{L^{\frac 43} L^{2}} \leq C \varepsilon^\sigma\|\d_{t} \bvp \|_{L^{\frac 43} L^{2}},
\ea
 where we used the fact $(3-r_{4})\a > \a >3$. At the same time, \eqref{0-integrals} also shows that $\bvp \cdot \Grad g_{\e}$ has zero integral mean. Employing \eqref{pro-div1}, similar arguments give
\ba
|I_7|& \leq\|\vr_\e\uu_\e \otimes \uu_\e \|_{L^{\frac{4}{3}} L^2} \|\Grad \calB_{\e}(\bvp \cdot \Grad g_{\e}) \|_{L^{4} L^{2}} \leq C \|\bvp \cdot\Grad g_{\e}\|_{L^{4} L^{2}} \\
& \leq C\|\Grad g_{\e}\|_{L^{q}}\|\bvp \|_{L^{4} L^{r_2}} \leq C \varepsilon^\sigma\|\Grad\bvp \|_{L^{4} L^{r_1}},\nn
\ea
and
\ba
|I_8|& \leq C\|\mathbb{S}(\Theta_\e, \Grad\vu) \|_{L^{2} L^2} \|\Grad \calB_{\e}(\bvp \cdot \Grad g_{\e}) \|_{L^{2} L^{2}} \leq C \|\bvp \cdot\Grad g_{\e}\|_{L^{2} L^{2}} \\
& \leq C\|\Grad g_{\e}\|_{L^{q}}\|\bvp \|_{L^{4} L^{r_2}} \leq C \varepsilon^\sigma\|\Grad\bvp \|_{L^{4} L^{r_1}},\nn
\ea
where we used 
\ba\nn
\|\mathbb{S}(\Theta_\e, \Grad\vu) \|_{L^{2} L^2}=\|\mu(\Theta_\e) \left( \Grad \vu_\e + \Grad^{\textup{T}} \vu_\e \right) \|_{L^{2} L^2}\leq C \|\Grad \vu_\e \|_{L^{2} L^2}\leq C.
\ea
Combining the properties of $\Theta_\e$ in \eqref{uniform bound of theta} and $\Grad F$ in \eqref{assumption of F}, we can find
\ba\nn
|I_9|& \leq C\|\Theta_\e \|_{L^\infty L^2} \|\Grad F\|_{L^\infty} \|\calB_{\e}(\bvp \cdot \Grad g_{\e}) \|_{L^{4} L^{2}} \leq C \|\bvp \cdot\Grad g_{\e}\|_{L^{4} L^{2}} \\
& \leq C\|\Grad g_{\e}\|_{L^{q}}\|\bvp \|_{L^{4} L^{r_2}} \leq C \varepsilon^\sigma\|\Grad\bvp \|_{L^{4} L^{r_1}}.
\ea
Additionally,
\ba\nn
|I_{10}|&\leq C\|\vr^0_\e \|_{L^\infty} \|\uu^0_\e\|_{L^2} \|\calB_{\e}(\bvp(0,x) \cdot \Grad g_{\e}) \|_{L^{2}} \leq C \|\bvp(0,x) \cdot\Grad g_{\e}\|_{L^{2}} \\
& \leq C\|\Grad g_{\e}\|_{L^{q}}\|\bvp(0,x) \|_{L^{r_2}} \leq C \varepsilon^\sigma\|\bvp (0,x)\|_{L^{r_2}}.
\ea

Finally, we put together the above estimates for $I_{i}$ $(i=1,\cdots,10)$ to conclude
\ba\nn
 |I^\varepsilon| \leq C\varepsilon ^\sigma \big(\|\partial _t \bvp \|_{L^{\frac43} L^{2}}
+\|\Grad \bvp\|_{L^{4} L^{r_1}}+\|\bvp (0,x)\|_{L^{r_2}} \big),
\ea
which yields our desired result \eqref{est-F-3D}.
\end{proof}

\subsection{Convergence of the nonlinear convective term}\label{Convergence of the nonlinear convective term}
 To pass to the limit as $\e\to 0$ for momentum equations in the homogenized domain, we shall establish the convergence of the nonlinear convective term $\widetilde \vr_\e\widetilde \uu_{\e}\otimes \widetilde \uu_{\e}$, where an  Aubin-Lions type argument (see Lemma 5.1 in \cite{Lions-com}) is needed:
\begin{lemma}\label{Aubin-Lions argument}
Let $g^n\to g$ weakly in $L^{p_1}(0,T;L^{p_2}(\Omega))$, $h^n \to h$ weakly in $L^{q_1}(0,T;L^{q_2}(\Omega))$,
where $1 \leq p_1,p_2 \leq +\infty$ and
\be\nn
\frac{1}{p_1}+\frac{1}{q_1}=\frac{1}{p_2}+\frac{1}{q_2}=1.
\ee
Additionally, we assume that
\be\nn
\frac{\d g^n}{\d t}\mbox{ is bounded in } L^1(0,T;W^{-m,1}(\Omega)) \mbox{ for some } m\geq 0 \mbox{ independent of } n,
\ee
and
\be\nn
\|h^n-h^n(\cdot+\xi,t)\|_{L^{q_1}(0,T;L^{q_2}(\Omega))}\to 0\ \ \mbox{as } |\xi|\to 0,
\text{ uniformly in } n.
\ee
Then there holds
\be\nn
g^n h^n \to gh \mbox{ in } {\cal D'}((0,T)\times \Omega).
\ee
\end{lemma}

 We observe that some uniform estimates for the time derivative can be deduced from Proposition \ref{moment-equa}. For any $\bsp\in C_c^\infty ((0,T)\times\Omega ;{\mathbb{R}^3})$ with $\dive\bsp=0$, Proposition \ref{moment-equa} shows
\ba\nn
\left| \l \partial _t (\widetilde \vr_\e\widetilde \uu_\e), \bsp \r \right| 
&\leq \int_0^T \int_\Omega \left| \widetilde \vr_\e\widetilde \uu_\e\otimes \widetilde \uu_\varepsilon :\Grad \bsp \right| + \left|\mathbb{S}(\Theta_\e, \Grad \widetilde\vu_\e):\Grad \bsp \right| + \left| \Theta_\e \Grad F \cdot \bsp \right| \dx\dt +| \l\GG_\varepsilon, \bsp \r|\\
&\leq \|\widetilde \vr_\e\widetilde \uu_\e\otimes \widetilde \uu_\varepsilon\|_{L^{\frac43} L^2}\|\Grad \bsp\|_{L^4 L^2} +\|\mathbb{S}(\Theta_\e, \Grad \widetilde\vu_\e)\|_{L^2 L^2}\|\Grad \bsp\|_{L^2 L^2}\\
&~~+\|\Theta_\e \Grad F\|_{L^2 L^2}\|\Grad \bsp\|_{L^2 L^2}+ C\e^\sigma\big(\|\partial_t \bsp \|_{L^{\frac43} L^2}+\|\Grad \bsp\|_{L^{4} L^{r_1}} ) \\
&\leq C\|\Grad \bsp\|_{L^4 L^2} + C\e^\sigma\big(\|\partial_t \bsp \|_{L^{\frac43} L^2}+\|\Grad \bsp\|_{L^{4} L^{r_1}} 
\big) \\
& \leq C\|\Grad \bsp\|_{L^{4} L^{r_1}} + C \e^\sigma\|\partial_t \bsp \|_{L^{\frac43} L^2},
\ea
where $\sigma>0$ and $r_{1} \in (2,3)$. Let $\mathbb P$ be the Leray-Helmholtz projection operator onto subspace of divergence free vector fields on $\Omega$. Thus we have the following decomposition
\ba\label{decomposition-u}
\mathbb P(\widetilde \vr_\e\widetilde \uu_\e) = (\widetilde \vr_\e \widetilde \uu_\e)^{(1)}
+ \e ^\sigma (\widetilde \vr_\e \widetilde \uu_\e)^{(2)},
\ea
where $\partial_t (\widetilde\vr_\e\widetilde \uu_\e) ^{(1)}$ is uniformly bounded in $L^{\frac{4}{3}}(0,T; W^{-1,r_{1}'}(\Omega)) $ and $(\widetilde \vr_\e\widetilde\uu_\e)^{(2)}$ is uniformly bounded in $L^4(0,T; L^{2}(\Omega))$, which implies 
\be\label{remainder1}
\e ^\sigma(\widetilde \vr_\e\widetilde\uu_\e)^{(2)}\to 0 \ \mbox{strongly in} \ L^4(0,T; L^{2}(\Omega;\R^3)).
\ee
Since $\widetilde \vr_\e\widetilde\uu_\e$ is uniformly bounded in $L^{\infty}(0,T; L^{2}(\Omega)) \cap L^2(0,T; L^6(\Omega))$ (see \eqref{est-rho u}), the same is true for $\mathbb P(\widetilde \vr_\e\widetilde\uu_\e)$. Consequently, $(\widetilde \vr_\e\widetilde\uu_\e)^{(1)} =\mathbb P(\widetilde \vr_\e\widetilde\uu_\e) - \e^\sigma(\widetilde \vr_\e\widetilde\uu_\e)^{(2)}$ is uniformly bounded in $L^4(0,T; L^{2}(\Omega))$. Using \eqref{decomposition-u}--\eqref{remainder1}, we have
\ba\label{conv-P(ru)(1)}
(\widetilde \vr_\e\widetilde\uu_\e)^{(1)} \to \overline{\mathbb P(\vr\uu)}\ \mbox{weakly in} \ L^4(0,T; L^{2}(\Omega;\R^3)).
\ea
Then we need to verify $\overline{\mathbb P(\vr\uu)}=\mathbb P(\vr\uu)$. Indeed, for any test function $\bsp\in C_c^\infty ((0,T)\times\Omega ;{\mathbb{R}^3})$ with $\dive\bsp=0$, we have
\ba\nn
\int_0^T \int_{\Omega} (\widetilde \vr_\e\widetilde\uu_\e)^{(1)}\cdot \bsp\dx\dt &= \int_0^T \int_{\Omega} \mathbb P(\widetilde \vr_\e\widetilde\uu_\e)\cdot \bsp\dx\dt- \int_0^T \int_{\Omega} \e^\sigma(\widetilde \vr_\e\widetilde\uu_\e)^{(2)}\cdot \bsp\dx\dt \\
&:=W_1+W_2.
 \ea
 By virtue of \eqref{Covergence of solu}, as $\e\to 0$, we have
 \ba\nn
W_1 &= \int_0^T \int_{\Omega} \mathbb P(\widetilde \vr_\e\widetilde\uu_\e)\cdot \bsp\dx\dt=\int_0^T \int_{\Omega}\widetilde\vr_\e \widetilde\uu_\e \cdot \bsp \dx\dt\\
&~{\to}\int_0^T \int_{\Omega}\vr\uu \cdot \bsp \dx\dt=\int_0^T \int_{\Omega} \mathbb P(\vr\uu) \cdot \bsp \dx\dt, 
 \ea
 and
 \ba\nn
|W_2| = \bigg|-\int_0^T \int_{\Omega} \e^\sigma(\widetilde \vr_\e\widetilde\uu_\e)^{(2)}\cdot \bsp\dx\dt\bigg|\leq \e^\sigma\|(\widetilde \vr_\e\widetilde\uu_\e)^{(2)}\|_{L^4(0,T; L^{2}(\Omega))}\cdot\|\bsp\|_{L^\frac43(0,T; L^{2}(\Omega))}\to 0.
 \ea
Thus we have
\ba\nn
& \int_0^T \int_{\Omega} (\widetilde \vr_\e\widetilde\uu_\e)^{(1)}\cdot \bsp\dx\dt \to \int_0^T \int_{\Omega} \mathbb P(\vr\uu) \cdot \bsp \dx\dt,~\mbox{as}~\e\to 0,
 \ea
which combined with \eqref{conv-P(ru)(1)} yields

\ba\label{conv-P(ru)(1')}
(\widetilde \vr_\e\widetilde\uu_\e)^{(1)} \to {\mathbb P(\vr\uu)} \quad \mbox{weakly in} \ L^4(0,T; L^{2}(\Omega;\R^3)).
\ea
Since $\partial_t (\widetilde\vr_\e\widetilde \uu_\e) ^{(1)}$ is uniformly bounded in $L^{\frac{4}{3}}(0,T; W^{-1,r_{1}'}(\Omega)) $,  $(\widetilde \vr_\e\widetilde\uu_\e)^{(1)}$ is uniformly bounded in $L^4(0,T; L^{2}(\Omega))$, together with the fact that $\widetilde\uu_\e$ converges weakly to $\uu$ in $L^2(0,T;W_0^{1,2}(\Omega))$ in \eqref{convergence}, an Aubin-Lions type argument in Lemma \ref{Aubin-Lions argument} yields
\ba\label{conv-ruu1}
(\widetilde \vr_\e\widetilde\uu_\e)^{(1)} \otimes \widetilde\uu_\e \to \mathbb P(\vr\uu) \otimes \uu \quad \mbox{in} \ \mathcal{D}'((0,T)\times \Omega).
\ea
Direct computation shows
\be\nn
\e^\sigma(\widetilde \vr_\e\widetilde\vu_\e)^{(2)}\otimes\widetilde\vu_\e \to {\bm 0} \textup{ strongly in }L^{4\over3}(0,T;L^{3\over2}(\Omega;\mathbb{R}^{3\times 3}))\cap L^{4}(0,T;L^{1}(\Omega;\mathbb{R}^{3\times 3})).
\ee
Thus we have
\be\nn
\mathbb P(\widetilde \vr_\e\widetilde\uu_\e) \otimes \widetilde\uu_\e \to \mathbb P(\vr\uu) \otimes \uu \quad \mbox{in} \ \mathcal{D}'((0,T)\times \Omega).
\ee
Together with the uniform estimates of $\mathbb P(\widetilde \vr_\e\widetilde\uu_\e)$ and $\widetilde\uu_\e$, we have
\be\label{nonlinear-covergence}
\mathbb P(\widetilde \vr_\e\widetilde\uu_\e) \otimes \widetilde\uu_\e \to \mathbb P(\vr\uu) \otimes \uu \quad \mbox{weakly in} \ L^{\frac43}(0,T;L^{\frac32}(\Omega;\mathbb{R}^{3\times 3})).
\ee

Since $\dive \widetilde\uu_\e=\dive \uu=0$, applying \eqref{nonlinear-covergence}, we find
\ba\label{norm-convergence}
&\int_0^T \int_{\Omega} \widetilde\vr_\e| \widetilde\uu_\e |^{2} \dx\dt = \int_0^T (\widetilde\vr_\e \widetilde\uu_\e, \widetilde\uu_\e)_{L^{2}(\Omega)}\dt \\
&~~=\int_0^T (\mathbb P(\widetilde\vr_\e \widetilde\uu_\e), \widetilde\uu_\e)_{L^{2}(\Omega)}\dt=\int_0^T \int_{\Omega} \mathbb P(\widetilde\vr_\e \widetilde\uu_\e) \cdot \widetilde\uu_\e \dx\dt\\
&~~~{\to}\int_0^T \int_{\Omega} \mathbb P(\vr\uu) \cdot \uu \dx\dt=\int_0^T ( \mathbb P(\vr\uu),\uu) _{L^{2}(\Omega)}\dt\\
&~~=\int_0^T (\vr\uu,\uu) _{L^{2}(\Omega)}\dt =\int_0^T \int_{\Omega} \vr| \uu|^{2} \dx\dt, ~\mbox{as}~\e\to 0.
 \ea
From \eqref{Covergence of rho} and $\eqref{Covergence of solu}_2$,  it follows that $\sqrt{\widetilde\vr_\e}\widetilde \uu_\varepsilon$ converges to $\sqrt{\vr}\uu$ weakly in $L^{2}((0,T)\times \Omega;\mathbb R^{3})$. Employing \eqref{norm-convergence}, we immediately have
\be\label{strong convergence 1}
\sqrt{\widetilde\vr_\e}\widetilde \uu_\varepsilon\to \sqrt{\vr}\uu \quad\mbox{strongly in } L^{2}((0,T)\times \Omega;\mathbb R^{3}),~\mbox{as}~ \e \rightarrow 0.
\ee
By \eqref{eatimates for u} and \eqref{Covergence of rho}, $\sqrt{\widetilde\vr_\e}\widetilde \uu_\varepsilon$ is uniformly bounded in $L^{2}(0,T;L^6(\Omega;\R^3))$. Hence, $\eqref{strong convergence 1}$ yields, for all $1\leq q<6$, 
\be\label{strong convergence 2}
\sqrt{\widetilde\vr_\e}\widetilde \uu_\varepsilon \to \sqrt{\vr}\uu \quad\mbox{strongly in } L^{2}(0,T;L^q(\Omega;\R^3)),~\mbox{as}~\e \rightarrow 0.
\ee 
Finally, we conclude that for all $1\leq \tilde q<3$,
 \ba\label{convective-conv}
\widetilde\vr_\e\widetilde\uu_\e\otimes \widetilde\uu_\e = \sqrt{\widetilde\vr_\e} \widetilde \uu_\varepsilon \otimes \sqrt{\widetilde\vr_\e}\widetilde \uu_\varepsilon \to \sqrt{\vr}\uu \otimes \sqrt{\vr} \uu = \vr\uu\otimes \uu \ \mbox{strongly in} \ L^{1}(0,T; L^{\tilde q}(\Omega;\R^{3\times 3})).
\ea
Combining \eqref{strong convergence 2}, the boundedness of density $\widetilde\vr_\e$ in \eqref{estimate for rho}, and strong convergence of density $\widetilde\vr_\e$ in \eqref{Covergence of rho}, we obtain the strong convergence of $\widetilde\uu_\e$ as follows
\begin{equation}\label{Strong conv-u}
\widetilde\uu_\e \to \uu \ \mbox{strongly in}\ L^2(0,T; L^{2}(\Omega;\mathbb R^3)).
\end{equation}

\subsection{Limit heat equation}\label{Limit heat equation}
To pass to the limit in the heat equation,  we first observe the fact for any $1\leq q < \infty$:
\ba\label{Covergence of kappa}
\|\kappa_\e - \kappa_f\|_{L^q}^q&=\int_{\Omega \setminus {{\Omega}}_\e} {|\kappa_s - \kappa_f|^q}\dx={|\kappa_s - \kappa_f|^q}\cdot |\Omega \setminus {{\Omega}}_\e |\to 0 \mbox{ as } \e \rightarrow 0.
\ea

 Based on the convergence results for temperature $\Theta_\e$ and velocity $\widetilde\uu_\e$, now we are ready to derive the limit heat equation. For any $\phi \in C^{\infty}_c ((0,T)\times\Omega)$, applying weakly convergence of $\Theta_\e$ in $\eqref{Covergence of solu}_3$ and strong convergence of $\kappa_\e$ in \eqref{Covergence of kappa}, we have
\ba\label{limit-temp-1}
\int_0^T\intO{ \kappa_\e \Grad \Theta_\e \cdot \Grad \phi }\dt\to \int_0^T\intO{ \kappa_f \Grad \Theta \cdot \Grad \phi }\dt, \mbox{ as } \e \rightarrow 0,
\ea
and applying weakly-(*) convergence of $\widetilde\uu_\e$ in $\eqref{Covergence of solu}_2$, we find
\ba\label{limit-temp-2}
\int_0^T\intO{ \widetilde\vu_\e \cdot \Grad F \phi }\dt\to \int_0^T\intO{ \vu \cdot \Grad F \phi }\dt, \mbox{ as } \e \rightarrow 0.
\ea
Applying \eqref{limit-temp-1}-\eqref{limit-temp-2}, passing $\e\to 0$ in \eqref{temp eq}, we can deduce that for any $\phi \in C^{\infty}_c ((0,T)\times\Omega)$,  there holds
\ba\label{limit-temp-equation}
\int_0^T\intO{ \kappa_f \Grad \Theta  \cdot \Grad \phi }\dt = \int_0^T\intO{ \vu  \cdot \Grad F \phi }\dt,
\ea
 which is exactly the weak formulation for heat equation $\eqref{Darcy}_{4}$.
\subsection{Strong convergence of temperature}\label{Strong convergence of temperature}
Actually, on the one hand, the elliptic structure of heat equation can furthermore guarantee higher regularity of temperature $\Theta$. On the other hand, the heat equation in homogenized domain $\O$ can show us the strong convergence of $\Theta_\e$ in some regularity spaces. Now we give specific derivation process to find convergence properties of $\Theta_\e$ in the following proposition.
\begin{proposition} \label{Strong convergence of theta} The temperature sequence $\{\Theta_\e\}_{\e>0}$ satisfy the following strong convergence:
\begin{equation} \nn
\Theta_\e \to \Theta \ \mbox{strongly in}\ L^2(0,T; W^{1,2}_0(\Omega)). 
\end{equation}
\end{proposition}

\begin{proof} 
From $\eqref{Covergence of solu}_3$, we have 
\ba\nn
\Theta_\e \to \Theta \ \mbox{weakly in}\ L^2(0,T; W^{1,2}_0(\Omega)). 
\ea
As a matter of fact, since $\Theta$ satisfies the limit problem \eqref{limit-temp-equation}, along with \eqref{assumption of F} and $\eqref{Covergence of solu}_3$, elliptic theory in \cite{LADUR}  gives higher regularities for $\Theta$ as follows
\begin{equation} \label{regularity of th}
\Theta \in L^2 (0,T; W^{2,2} \cap W^{1,2}_0(\Omega)).
\end{equation}

Since $C^{\infty}_c ((0,T)\times\Omega)$ is dense in $L^2 (0,T; W^{2,2} \cap W^{1,2}_0(\Omega))$, we take $\phi=\Theta$ as a test function in \eqref{limit-temp-equation} to find 
\ba\label{temp eq-3}
\int_0^T\intO{ \kappa_f \Grad\Theta \cdot \Grad \Theta}\dt =\int_0^T \intO{\vu \cdot \Grad F \Theta}\dt.
\ea
Taking $\phi=\Theta \in L^2 (0,T; W^{2,2} \cap W^{1,2}_0(\Omega))$ in \eqref{temp eq} and $\phi=\Theta_\e \in L^2(0,T; W^{1,2}_0(\Omega))$ in \eqref{temp eq} as test functions respectively, we have
\ba\label{temp eq-4}
\int_0^T\intO{ \kappa_\e\Grad \Theta_\e \cdot \Grad \Theta}\dt= \int_0^T\intO{ \widetilde\vu_\e \cdot \Grad F \Theta }\dt,
\ea
\ba\label{temp eq-5}
\int_0^T\intO{ \kappa_\e \Grad \Theta_\e \cdot \Grad \Theta_\e}\dt= \int_0^T\intO{ \widetilde\vu_\e \cdot \Grad F \Theta_\e }\dt.
\ea

Utilizing the equations from \eqref{temp eq-3} to \eqref{temp eq-5}, we find
\ba\label{strong-conver of theta}
\begin{split}
 &\int_0^T\intO{ \kappa_\e |\Grad \Theta_\e - \Grad \Theta|^2 }\dt\\
& ~~=\int_0^T \intO{ \kappa_{\e}  \Grad \Theta_\e \cdot   \Grad \Theta_\e  }\dt-2\int_0^T\intO{ \kappa_{\e}  \Grad \Theta_\e \cdot   \Grad \Theta  }\dt+\int_0^T\intO{ \kappa_{f}  \Grad \Theta \cdot   \Grad \Theta  }\dt\\
& ~~\quad+\int_0^T\intO{ (\kappa_{\e}-\kappa_{f})  \Grad \Theta \cdot   \Grad \Theta  }\dt\\
& ~~=\int_0^T\intO{ \widetilde\vu_\e \cdot \Grad F \Theta_\e }\dt-2\intO{ \widetilde\vu_\e \cdot \Grad F \Theta }\dt+\int_0^T\intO{\vu \cdot \Grad F \Theta}\dt\\
&~~\quad+\int_0^T\intO{ (\kappa_{\e}-\kappa_{f})  \Grad \Theta \cdot   \Grad \Theta  }\dt\\
& ~~=\int_0^T\intO{ \widetilde\vu_\e \cdot \Grad F (\Theta_\e-\Theta) }\dt+\int_0^T\intO{(\vu- \widetilde\vu_\e) \cdot \Grad F \Theta}\dt\\
&~~\quad+\int_0^T\intO{ (\kappa_{\e}-\kappa_{f})  \Grad \Theta \cdot   \Grad \Theta  }\dt\\
&~~:=K_1+K_2+K_3,
\end{split}
\ea
with
\ba\nn
&K_1=\int_0^T\intO{ \widetilde\vu_\e \cdot \Grad F (\Theta_\e - \Theta )} \dt,\\
&K_2=\int_0^T\intO{(\vu- \widetilde\vu_\e) \cdot \Grad F  \Theta }\dt,\\
&K_3=\int_0^T\intO{(\kappa_\e-\kappa_f )\Grad \Theta \cdot\Grad \Theta}\dt.
\ea

Moreover, using the weak convergence of $\Theta_\e$ in $\eqref{Covergence of solu}_3$ and the strong convergence of $ \widetilde\vu_\e$ in \eqref{Strong conv-u}, we deduce
\ba\label{time-space1}
K_1=\int_0^T \intO{ \widetilde\vu_\e \cdot \Grad F (\Theta_\e - \Theta )}\dt \to 0.
\ea
By  \eqref{assumption of F}, \eqref{Strong conv-u} and \eqref{regularity of th} we obtain
\ba\label{time-space2}
K_2=\int_0^T\intO{(\vu- \widetilde\vu_\e) \cdot \Grad F  \Theta }\dt\to 0.
\ea
For the last term on the right side of \eqref{strong-conver of theta}, by \eqref{Covergence of kappa} and \eqref{regularity of th}, we have 
\ba\label{time-space3}
 |K_3|\leq C\|\kappa_\e-\kappa_f \|_{L^{q^*}}\| \Grad \Theta\|_{ L^2L^{r_1}}\| \Grad \Theta \|_{L^2L^2}\leq C\|\kappa_\e-\kappa_f \|_{L^{q^*}}\to 0,
 \ea
 where $r_1\in(2,3)$ and $q^*\in(6,+\infty)$ are given in \eqref{r1}.

To conclude, we combine \eqref{time-space1}--\eqref{time-space3} to derive
\ba
 \int_0^T\intO{|\Grad \Theta_\e - \Grad \Theta|^2 }\dt\leq C\int_0^T\intO{\kappa_\e|\Grad \Theta_\e - \Grad \Theta|^2 } \dt\to 0, \mbox{ as } \e \rightarrow 0,
\ea
which with Poincar\'e inequality leads to
\begin{equation} \nn
\Theta_\e \to \Theta \ \mbox{strongly in}\ L^2(0,T; W^{1,2}_0(\Omega)).
\end{equation}
\end{proof}

\subsection{Limit momentum equations}\label{Limit momentum equations}
 After the establishment of these important regularities and convergences related to $(\widetilde\vr_\e, \widetilde\vu_\e, \Theta_\e)$, we are now in a position to derive the limit momentum equations.
 
 For each $\bvp \in C_{c}^{\infty}([0,T)\times\Omega;\R^{3}), \ \dive \bvp =0$, using \eqref{Covergence of solu}, we have
\ba\label{conv-1}
\int_0^T \int_\Omega (\widetilde\vr_\e \widetilde\uu_\e) \cdot \partial_{t}\bvp \dx\dt \to
\int_0^T \int_\Omega (\vr\uu) \cdot \partial_{t}\bvp \dx\dt, \mbox{ as } \e \rightarrow 0.
\ea
By \eqref{convective-conv}, we immediately have 
\ba\label{conv 2}
\int_0^T \int_{\Omega} \widetilde{\vr}_\e\widetilde{\uu}_\varepsilon \otimes \widetilde{\uu}_\varepsilon :\Grad \bvp \dx\dt \to\int_0^T \int_{\Omega}\vr\uu \otimes\uu:\Grad \bvp \dx\dt, \mbox{ as } \e \rightarrow 0.
\ea
Since $\mu$ is Lipschitz continuous on $\mathbb R$, using Proposition \ref{Strong convergence of theta}, we have
\ba\nn
\mu({\Theta_\e})\to\mu({\Theta})~ \mbox{strongly in}\ L^2((0,T)\times\Omega), \mbox{ as } \e \rightarrow 0.
\ea
Then we make use of the weak convergence for $\uu_\e$ in \eqref{Covergence of solu} to derive 
\ba\label{conv 3}
&\int_0^T \int_{\Omega} \mathbb{S}(\Theta_\e, \Grad \widetilde\vu_\e):\Grad \bvp \dx\dt=\int_0^T \int_{\Omega} \mu(\Theta_\e) \left( \Grad \widetilde\vu_\e + \Grad^{\rm T} \widetilde\vu_\e \right):\Grad \bvp \dx\dt\\
& \quad\to\int_0^T \int_{\Omega} \mu(\Theta) \left( \Grad \vu + \Grad^{\rm T} \vu \right):\Grad \bvp \dx\dt=\int_0^T \int_{\Omega} \mathbb{S}(\Theta, \Grad \vu):\Grad \bvp \dx\dt, \mbox{ as } \e \rightarrow 0.
\ea
Moreover, the weak convergence of $\Theta_\e$ in \eqref{Covergence of solu} gives
\ba\label{conv 4}
\int_0^T \int_{\Omega} \Theta_\e \Grad F \cdot \bvp \dx\dt \to\int_0^T \int_{\Omega} \Theta \Grad F \cdot \bvp \dx\dt, \mbox{ as } \e \rightarrow 0.
\ea
Since initial value conditions \eqref{initial-rho-u} shows $\widetilde\vr_{\e}^{0}\widetilde\uu_\e^0 \to \vr_0\uu_0 \ \mbox{strongly in} \ L^2(\Omega)$, we find
\ba\label{conv-5}
\int_\Omega (\widetilde\vr_\e^0\widetilde \uu_\e^0) \cdot \bvp(0,x) \,\dx \to \int_\Omega (\vr_0\uu_0) \cdot \bvp(0,x) \dx, \mbox{ as } \e \rightarrow 0.
\ea
Additionally, \eqref{est-F-3D} implies
\ba\label{limit-GG}
\l\GG_\varepsilon, \bvp \r\to 0, \mbox{ as } \e \rightarrow 0.
\ea

Using \eqref{conv-1}--\eqref{limit-GG}, passing $\e\to 0$ in \eqref{momentum eq-1}, we can deduce that for each $\bvp \in C_{c}^{\infty}([0,T)\times\Omega;\R^{3}), \ \dive \bvp =0$,
\ba\label{limit-momentum-equ}
&\int_{0}^{T} \int_\Omega \vr\uu \cdot \d_{t} \bvp +\vr\uu\otimes \uu : \Grad \bvp \dx\dt-\int_{0}^{T} \int_\Omega \mathbb{S}(\Theta, \Grad \vu) :\Grad \bvp\dx\dt \\
&=\int_{0}^{T} \int_\Omega \Theta \Grad F \cdot \bvp\dx\dt+ \int_\Omega \vr_0\uu_0 \cdot \bvp(0,x) \dx.
\ea

Combining \eqref{estimate for rho}, \eqref{Strong conv-u}, \eqref{limit-temp-equation}, \eqref{limit-momentum-equ}, Proposition \ref{compact-density} and Proposition \ref{Strong convergence of theta}, we complete the proof of Theorem \ref{Main theorem}.

	\paragraph{Conflict of interest}
	The authors declare no conflict of interest in this paper.
	
	
	

\end{document}